\documentclass[11pt]{article}
\usepackage{amsmath, amscd, amssymb, latexsym, epsfig, url,color, amsthm}
\usepackage{enumitem}
\usepackage{tikz}
\usetikzlibrary{calc}
\usepackage{pstricks,graphicx,subfig}
\numberwithin{equation}{section}

\newtheorem{theorem}{Theorem}[section]
\newtheorem{proposition}[theorem]{Proposition}
\newtheorem{question}[theorem]{Question}

\newtheorem{lemma}[theorem]{Lemma}

\theoremstyle{definition}
\newtheorem{definition}[theorem]{Definition}

\newtheorem{construction}[theorem]{Construction}
\newtheorem{remark}[theorem]{Remark}

\DeclareMathOperator\lk{\mathrm{lk}}
\DeclareMathOperator\st{\mathrm{st}}
\DeclareMathOperator{\intr}{\mathrm{int}}

\newcommand{\field}{{\bf k}}

\newcommand{\R}{{\mathbb R}}

\newcommand{\Z}{{\mathbb Z}}

\newcommand{\Sp}{\mathbb{S}}

\title{Ear Decomposition and Balanced Neighborly Simplicial Manifolds}
\author{Hailun Zheng\\
	\small Department of Mathematics \\[-0.8ex]
	\small University of Michigan\\[-0.8ex]
	\small Ann Arbor, MI, 48109, USA\\[-0.8ex]
	\small \texttt{hailunz@umich.edu}
}
\begin{document}
	\maketitle
	\begin{abstract}
		We find the first non-octahedral balanced 2-neighborly 3-sphere and the balanced 2-neighborly triangulation of the lens space $L(3,1)$. Each construction has 16 vertices. We show that there exists a balanced 3-neighborly non-spherical 5-manifold with 18 vertices. We also show that the rank-selected subcomplexes of a balanced simplicial sphere do not necessarily have an ear decomposition.
	\end{abstract}
	\section{Introduction}
	A simplicial complex is called $k$-neighborly if every subset of vertices of size at most $k$ is the set of vertices of one of its faces. Neighborly complexes, especially neighborly polytopes and spheres, are interesting objects to study. In the seminal work of McMullen \cite{M} and Stanley \cite{St}, it was shown that in the class of polytopes and simplicial spheres of a fixed dimension and with a fixed number of vertices, the cyclic polytope simultaneously maximizes all the face numbers. The $d$-dimensional cyclic polytope is $\lfloor \frac{d}{2}\rfloor$-neighborly. Since then, many other classes of neighborly polytopes have been discovered. We refer to \cite{G}, \cite{P} and \cite{Sh} for examples and constructions of neighborly polytopes. Meanwhile, the notion of neighborliness was extended to other classes of objects: for instance, neighborly cubical polytopes were defined and studied in \cite{JR}, \cite{JZ}, and \cite{SZ}, and neighborly centrally symmetric polytopes and spheres were studied in \cite{B}, \cite{DT}, \cite{J} and \cite{LN}. 
	
	In this paper we discuss a similar notion for balanced simplicial complexes. Balanced complexes were defined by Stanley in \cite{St2}, where they were called completely balanced. A $(d-1)$-dimensional simplicial complex is called balanced if its graph is $d$-colorable. For instance, the barycentric subdivision of regular CW complexes and order complexes are balanced. We say that a balanced simplicial complex is \emph{balanced $k$-neighborly} if every set of $k$ or fewer vertices with \textit{distinct} colors forms a face. For example, if $\Delta_1$ and $\Delta_2$ are balanced $k$-neighborly spheres, then the join $\Delta= \Delta_1*\Delta_2$ is also a balanced neighborly $k$-sphere, and we call $\Delta$ \emph{join-decomposable}. However, apart from the cross-polytopes, it is not known whether other \emph{join-indecomposable} balanced $k$-neighborly polytopes or spheres exist. To the best of our knowledge, no examples of such objects appear in the current literature, even for $k=2$. As for balanced 2-neighborly manifolds, one such construction that triangulates the sphere bundle is given in \cite{KN}; it is also a minimal balanced triangulation of the underlying topological space.

    This more or less explains why so far there is even no plausible sharp upper bound conjecture for balanced spheres or manifolds. The goal of this paper is to partially remedy this situation by searching for balanced neighborly spheres and manifolds of lower dimensions. It turns out that even in the lower dimensional cases balanced neighborly spheres or manifolds with a given number of vertices do not always exist.
    \begin{itemize}
    	\item The octahedral 3-sphere is the only balanced $2$-neighborly $3$-sphere with less than $16$ vertices.
    	\item There is a unique balanced 2-neighborly 4-sphere with 15 vertices, known as $^4 15^5_2$ in \cite{KL}.
    	\item There exists a balanced 3-neighborly non-spherical 5-manifold with 18 vertices.
    	\item There are two constructions of balanced 2-neighborly 3-manifolds with 16 vertices; one triangulates the sphere, and the other triangulates the lens space $L(3,1)$. 
    \end{itemize}
	
	In a different direction, it is also interesting to ask whether every rank-selected subcomplex of a balanced simplicial polytope or sphere has a convex ear decomposition. This statement, if true, would imply that rank-selected subcomplexes of balanced simplicial polytopes possess certain weak Lefschetz properties, see Theorem 3.9 in \cite{Sw}. As a consequence, it would also provide an alternative proof of the balanced Generalized Lower Bound Theorem, see Theorem 3.3 and Remark 3.4 in \cite{MS}. We present an example giving a negative answer to this question for 3-dimensional spheres.
	
	The structure of this manuscript is as follows. In Section 2, after reviewing basic definitions, we establish basic properties of balanced neighborly spheres; in particular, we prove that for some values of $f_0$, such spheres cannot exist. In Section 3 we discuss how to find balanced $k$-neighborly $(2k-1)$-manifolds from less neighborly balanced $(2k-2)$-spheres, for $k=2,3$. In Section 4, we construct a balanced 2-neighborly 3-sphere with 16 vertices. In Section 5, we present the balanced 2-neighborly triangulation of $L(3,1)$ with 16 vertices. In Section 6 we provide a way to construct balanced spheres whose rank-selected subcomplex does not have an ear decomposition.
	
	\section{Basic properties of balanced neighborly spheres}
	A \textit{simplicial complex} $\Delta$ with vertex set $V$ is a collection of subsets
	$\sigma\subseteq V$, called \textit{faces}, that is closed under inclusion, and such that for every $v \in V$, $\{v\} \in \Delta$. For $\sigma\in \Delta$, let $\dim\sigma:=|\sigma|-1$ and define the \textit{dimension} of $\Delta$, $\dim \Delta$, as the maximum dimension of the faces of $\Delta$. A \emph{facet} is a maximal face under inclusion. We say that a simplicial complex $\Delta$ is \textit{pure} if all of its facets have the same dimension.

	If $\Delta$ is a simplicial complex and $\sigma$ is a face of $\Delta$, the \textit{star} of $\sigma$ in $\Delta$ is $\st_\Delta \sigma:= \{\tau \in\Delta: \sigma\cup\tau\in\Delta \}$. We also define the \textit{link} of $\sigma$ in $\Delta$ as $\lk_\Delta \sigma:=\{\tau-\sigma\in \Delta: \sigma\subseteq \tau\in \Delta\}$, and the \textit{deletion} of a subset of vertices $W$ from $\Delta$ as $\Delta\backslash W:=\{\sigma\in\Delta:\sigma\cap W=\emptyset\}$. If $\Delta_1$ and $\Delta_2$ are simplicial complexes on disjoint vertex sets, then the join of $\Delta_1$ and $\Delta_2$, denoted $\Delta_1*\Delta_2$, is the simplicial complex with vertex set $V(\Delta_1)\cup V(\Delta_2)$ whose faces are $\{\sigma_1\cup\sigma_2:\sigma_1\in \Delta_1,\sigma_2\in\Delta_2\}$.
	
	If $\Delta$ is a pure $(d-1)$-dimensional complex such that every $(d-2)$-dimensional face of $\Delta$ is contained in at most two facets, then the \emph{boundary complex} of $\Delta$ consists of all $(d-2)$-dimensional faces that are contained in exactly one facet, as well as their subsets. A simplicial complex $\Delta$ is a \textit{simplicial sphere} (resp. \emph{simplicial ball}) if the geometric realization of $\Delta$ is homeomorphic to a sphere (resp. ball). The boundary complex of a simplicial $d$-ball is a simplicial $(d-1)$-sphere. A simplicial sphere is called \textit{polytopal} if it is the boundary complex of a convex polytope. For instance, the boundary complex of an octahedron is a polytopal sphere; we will refer to it as an octahedral sphere. 
	
	For a fixed field or group $\field$, we say that $\Delta$ is a $(d-1)$-dimensional \textit{$\field$-homology sphere} if $\tilde{H}_i(\lk_\Delta \sigma;\field)\cong \tilde{H}_i(\mathbb{S}^{d-1-|\sigma|};\field)$ for every face $\sigma\in\Delta$ (including the empty face) and $i\geq -1$. A \emph{homology $d$-ball} (over $\field$) is a $d$-dimensional simplicial complex $\Delta$ such that (i) $\Delta$ has the same homology as the $d$-dimensional ball, (ii) for every face $F$, the link of $F$ has the same homology as the $(d-|F|)$-dimensional ball or sphere, and (iii) the boundary complex $\partial\Delta$ is a homology $(d-1)$-sphere. The classes of simplicial $(d-1)$-spheres and homology $(d-1)$-spheres coincide when $d\leq 3$. From now on all homology are computed with coefficients in $\Z$ and we will omit it from our notation. 
	
	Next we define a special structure that exists in some pure simplicial complexes.
	\begin{definition}
		An \textit{ear decomposition} of a pure $(d-1)$-dimensional simplicial complex $\Delta$ is an ordered sequence $\Delta_1,\Delta_2,\cdots, \Delta_m$ of pure $(d-1)$-dimensional subcomplexes of $\Delta$ such that:
		\begin{enumerate}
			\item $\Delta_1$ is a simplicial $(d-1)$-sphere, and for each $j=2,3,\cdots,m$, $\Delta_j$ is a
			simplicial $(d-1)$-ball.
			\item For $2\leq j\leq m$, $\Delta_j\cap (\cup_{i=1}^{j-1}\Delta_i)=\partial\Delta_j$.
			\item $\cup_{i=1}^m \Delta_i=\Delta.$
		\end{enumerate}
		We call $\Delta_1$ the \textit{initial complex}, and each $\Delta_j$, $j\geq 2$, an \textit{ear of this decompostion}. Notice that this definition is more general than Chari's original definition of a \emph{convex ear decomposition}, see \cite[Section 3.2]{C}, where the $\Delta_i$'s are required to be subcomplexes of the boundary complexes of polytopes. In particular, if a complex has no ear decomposition, then it has no convex ear decomposition. However, by the Steinitz theorem, all simplicial 2-spheres are polytopal, and hence also all simplicial 2-balls can be realized as subcomplexes of the boundary complexes of 3-dimensional polytopes. So for 2-dimensional simplicial complexes, the notion of an ear decomposition coincides with that of a convex ear decomposition.
	\end{definition}
	
	A $(d-1)$-dimensional simplicial complex $\Delta$ is called \textit{balanced} if the graph of $\Delta$ is $d$-colorable, or equivalently, there is a coloring map $\kappa: V\to [d]$ such that $\kappa(x)\neq \kappa(y)$ for any edge $\{x,y\}\in\Delta$. Here $[d]=\{1,2,\cdots,d\}$ is the set of colors. We denote by $V_i$ the set of vertices of color $i$. A balanced simplicial complex is called \textit{balanced $k$-neighborly} if every set of $k$ or fewer vertices with distinct colors forms a face. We say $e$ is a \emph{missing colored edge} if $e\notin \Delta$ and the vertices of $e$ have distinct colors. For $S\subseteq [d]$, the subcomplex $\Delta_S:=\{F\in \Delta: \kappa(F)\subseteq S\}$ is called the \textit{rank-selected subcomplex} of $\Delta$. We also define the \textit{flag $f$-vector} $(f_S(\Delta):S\subseteq [d])$ and the \textit{flag $h$-vector} $(h_S(\Delta):S\subseteq [d])$ of $\Delta$, respectively, by letting  $f_S(\Delta):=\#\{F\in \Delta: \kappa(F)=S\}$, where $f_{\emptyset}(\Delta)=1$, and $h_S(\Delta):=\sum_{T\subseteq S}(-1)^{\#S-\#T}f_T(\Delta)$. The usual $f$-numbers and $h$-numbers can be recovered from the relations $f_{i-1}(\Delta)=\sum_{\#S=i}f_S(\Delta)$ and $h_{i}(\Delta)=\sum_{\#S=i}h_S(\Delta)$. 
	
	In the remainder of this section, we establish some restrictions on the possible size of color sets of balanced neighborly spheres.
		\begin{lemma}\label{prop: d=2k k-neighborly}
			Let $\Delta$ be a balanced $k$-neighborly homology $(2k-1)$-sphere. Then $\Delta$ has the same number of vertices of each color. In particular, $f_0(\Delta)=2k\ell$ for some $\ell\geq 2$.
		\end{lemma}
		\begin{proof}
			Let $W\subseteq [2k]$ be an arbitrary subset of the set of the colors with $|W|=k$. Since $\Delta$ is balanced $k$-neighborly, $\Delta_W$ is also balanced $k$-neighborly, and hence $\Delta_W$ is the join of $k$ color sets of colors in $W$, each considered as a 0-dimensional complex. By the definition of the join and the flag $h$-numbers, we have $f_{U\cup \{i\}}(\Delta)=f_{U}(\Delta)f_{\{i\}}(\Delta)$ and hence $h_{U\cup \{i\}}(\Delta)=h_{U}(\Delta)h_{\{i\}}(\Delta)$ for all $i\in W$, $U\subset W$ and $i\notin U$. Therefore,
			\[\prod_{i\in W}(|V_i|-1)=\prod_{i\in W}h_{\{i\}}(\Delta)=h_{W}(\Delta)\stackrel{(*)}{=}h_{[2k]\backslash W}(\Delta)=\prod_{i\in[2k]\backslash W}h_{\{i\}}(\Delta)=\prod_{i\in[2k]\backslash W}(|V_{i}|-1),\]
			where $(*)$ follows from the Dehn-Sommerville relations.
			Since $W$ is an arbitrary $k$-subset of $[2k]$, it follows that each color set in $\Delta$ must have the same size.
		\end{proof}
		\begin{remark}\label{rmk: extension from 3-sphere to 3-manifold}
			Lemma \ref{prop: d=2k k-neighborly} not only holds for homology $(2k-1)$-spheres but also for orientable homology $(2k-1)$-manifolds. Indeed by replacing the flag $h$-numbers with the flag $h''$-numbers (see \cite{JKMNS} for definition), Theorem 4.1 in \cite{JKMNS} gives $h''_W(\Delta)=h''_{[2k]\backslash W}(\Delta)$, which further implies that $h_W(\Delta)=h_{[2k]\backslash W}(\Delta)$ since both $W$ and $[2k]\backslash W$ are of size $k$. The rest of the proof is the same.
			
			Unfortunately, the above lemma is not sufficient to tell whether a balanced $k$-neighborly homology $(2k-1)$-sphere or manifold with $2k\ell$ vertices can exist for given $k, \ell\geq 2$.
		\end{remark}

		\section{Balanced neighborly $(d-1)$-manifolds with $3d$ vertices}
		In this section, we consider balanced $\lfloor \frac{d}{2}\rfloor$-neighborly $(d-1)$-manifolds (for $d=3,4,5$) with each color set of size 3. We begin with the following lemma.
		\begin{lemma}\label{lm: intersection of links is sphere}
			Let $d\geq 4$. If $\Delta$ is a balanced homology $(d-1)$-sphere and $V_d=\{v_1,v_2,v_3\}$ is the set of vertices of color $d$, then $\lk_\Delta v_i\cap \lk_\Delta v_j$ is a homology $(d-2)$-ball  for any $1\leq i<j\leq 3$, and $\cap_{k=1}^{3}\lk_\Delta v_k$ is a homology $(d-3)$-sphere. 
		\end{lemma}
		\begin{proof}
			Let $\{i, j, k\}=[3]$ be distinct, $\Sigma=\lk_\Delta v_i\cap \lk_\Delta v_j$ and $\Gamma=\cap_{k=1}^{3}\lk_\Delta v_k$. 
			We first prove that $\Sigma$ and $\Gamma$ have the same homology as a simplicial $(d-2)$-ball and simplicial $(d-3)$-sphere respectively. Since each $(d-2)$-face of $\Delta$ is contained in exactly 2 facets, it follows that $\lk_\Delta v_i\cup \lk_\Delta v_j=\Delta_{[d-1]}$. By the Mayer-Vietoris sequence, for any $n\geq 0$,
			\begin{equation}\label{eq:1}
			\cdots\to H_{n+1}(\Delta_{[d-1]})\to H_n(\Sigma)\to H_n(\lk_\Delta v_i)\oplus H_n(\lk_\Delta v_j)\to H_{n}(\Delta_{[d-1]})\to\cdots.
			\end{equation}
			Note that $\Delta_{[d-1]}$ is a deformation retract of $\Delta$ minus three points, hence $\beta_{d-2}(\Delta_{[d-1]})=2$ and $\beta_k(\Delta_{[d-1]})=0$ for $0\leq k\leq d-3$. We conclude from (\ref{eq:1}) that $\beta_k(\Sigma)=0$ for all $k\geq 0$. Since $\lk_\Delta v_k\cup \Sigma=\Delta_{[d-1]}$ and $\lk_\Delta v_k\cap \Sigma=\Gamma$, by the Mayer-Vietoris sequence we obtain
			\[	\cdots\to H_{n+1}(\Delta_{[d-1]})\to H_n(\Gamma)\to H_n(\lk_\Delta v_k)\oplus H_n(\Sigma)\to H_{n}(\Delta_{[d-1]}) \to\cdots.\]
			Hence $\beta_{d-3}(\Gamma)=1$ and $\beta_{k}(\Gamma)=0$ for $0\leq k\leq d-4$. 
			
			Next, for any $\tau\in \Gamma$, we have $\lk_\Sigma \tau=\lk_{\lk_\Delta \tau} v_i \cap \lk_{\lk_\Delta \tau} v_j$ and $\lk_\Gamma \tau=\cap_{i=1}^{3}\lk_{\lk_\Delta \tau} v_i$. Since $\lk_\Delta \tau$ is a balanced homology $(d-1-|\tau|)$-sphere, using the same argument as above, we may show that $\lk_\Sigma \tau$ and $\lk_\Gamma\tau$ have the same homology as a $(d-2-|\tau|)$-ball and $(d-3-|\tau|)$-sphere respectively. Therefore $\Gamma$ is a homology $(d-3)$-sphere. Finally, for any interior face $\sigma$ of $\Sigma$, $\lk_\Sigma \sigma=\lk_{\lk_\Delta v_i} \sigma=\lk_{\lk_\Delta v_j} \sigma$, and hence $\lk_\Sigma \sigma$ is a homology sphere. By definition we conclude that $\Sigma$ is a homology $(d-2)$-ball.
		\end{proof}
		\begin{remark}\label{rm: 1}
			The complex $\Gamma$ in Lemma \ref{lm: intersection of links is sphere} is not balanced, since $\Gamma$ is $(d-1)$-colorable instead of being $(d-2)$-colorable. 
		\end{remark}
	
		\begin{proposition}\label{prop: 2-nbly 3-sphere 12 ver}
			The only balanced 2-neighborly 3-manifold with 12 vertices triangulates the nonorientable $\Sp^2$-bundle over $\Sp^1$.
		\end{proposition}
		
		\begin{proof}
			\begin{figure}
					\centering
					\includegraphics[scale=0.9]{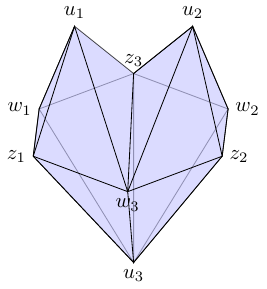}
					\hspace{20mm}
					\includegraphics[scale=0.9]{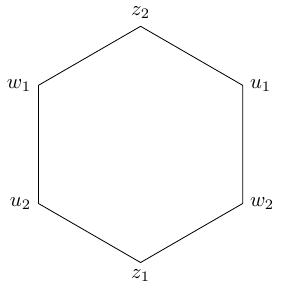}
					\caption{Left: triangulation of the vertex link $\lk_\Delta v_i$ for $v_i\in V_4$, where $\{u_1,u_2,u_3\}$, $\{w_1,w_2,w_3\}$ and $\{z_1,z_2,z_3\}$ are the three other color sets. Right: the complex $\Sigma$.}
					\label{fig: 12-vertex}
			\end{figure}
			Let $\Delta$ be a balanced 2-neighborly 3-manifold with 12 vertices. Its $f$-vector is $f(\Delta)=(1,12, 54, 84, 42)$. By Lemma \ref{prop: d=2k k-neighborly} and Remark \ref{rmk: extension from 3-sphere to 3-manifold}, each color set of $\Delta$ has three vertices. We let $V_4=\{v_1,v_2,v_3\}$ be the set of vertices of color 4. Since $\Delta$ is balanced 2-neighborly, each $\lk_\Delta v_i$ is a 2-sphere with 9 vertices, its $f$-vector is (1,9,21,14). Furthermore, the balancedness of $\Delta$ implies that every vertex $u\in \lk_\Delta v_i$ has $\deg_{\lk_\Delta v_i} u=4$ or 6. If $x$ is the number of vertices of degree 6 in $\lk_\Delta v_i$, then 
			\[4(9-x)+6x=\sum_{u\in \lk_\Delta v_i}\deg(\lk_{\lk_\Delta v_i}u)=2 f_1(\lk_\Delta v_i)=42,\]
			and hence $x=3$. A balanced 2-sphere with 9 vertices, 3 of which have degree 6, is unique up to isomorphism, as shown in Figure \ref{fig: 12-vertex}. Hence all vertex links in $\Delta$ are combinatorially isomorphic.
			
			Since $f_i(\Delta_{[3]})=f_i(\Delta)-\sum_{j=1}^3 f_{i-1}(\lk_\Delta v_j)$ for all $i\leq 2$, we have $f(\Delta_{[3]})=(1, 9, 27, 21)$. Let $i,j,k\in [3]$ be distinct and $\Sigma:=\lk_\Delta v_i\cap \lk_\Delta v_j$. Any facet $F$ of $\Sigma$ are 2-dimensional, for otherwise if $F$ is an edge, then $\lk_\Delta F$ is either a 4-cycle or 6-cycle, where in both cases $v_i$ and $v_j$ share at least one common neighbor $w$, i.e., $F\cup \{w\}\in \Sigma$, a contradiction. Similarly, the facet cannot be 0-dimensional. Also the facets of $\Sigma$ do not belong to $\lk_\Delta v_k$. Hence $\Sigma$ is a pure 2-dimensional subcomplex of $\lk_\Delta v_i$ with 9 vertices and $f_2(\Delta_{[3]})-f_2(\lk_\Delta v_k)=7$ triangles.
			
			On the other hand, for any $u\in \Sigma$, the vertex link $\lk_\Delta u$ is isomorphic to Figure \ref{fig: 12-vertex} and $v_i, v_j\in \lk_\Delta u$. The intersection of links of arbitrary two vertices of the same color has the following property 
			\begin{center}
				$(*)$: \quad$\lk_\Sigma u=\lk_{\lk_\Delta u} v_i\cap \lk_{\lk_\Delta u} v_j$ is either an edge, or a path of length 3.
			\end{center}
			
			Since $\Sigma$ is 2-dimensional, each connected component of $\Sigma$ has at least 3 vertices. If there are three components, then $\Sigma$ is the disjoint union of three triangles, contradicting that $f_2(\Sigma)=7$. Otherwise, if $\Sigma$ is connected, then by observation $(*)$ we have that $\Sigma$ is a triangulated 2-manifold (with boundary). Since $21=3f_2(\Sigma)=\sum_{u\in\Sigma} f_1(\lk_\Sigma u) $ and by observation $(*)$ $f_1(\lk_\Sigma u)\in \{1,3\}$, the links of three vertices in $\Sigma$ are single edges, while the rest are paths of length 3. However, enumeration based on observation $(*)$ yields that there is no such complex $\Sigma$.
			
			The last case is that $\Sigma$ has two connected components. From observation $(*)$ we see that each component cannot have 4 or 5 vertices. If one component is the triangle, then the other component (as a 6-triangle subcomplex of the 9-vertex balanced 2-sphere) must be the triangulated annulus as shown in Figure 1. By symmetry $\lk_\Delta v_i\cap \lk_\Delta v_k$ and $\lk_\Delta v_j\cap \lk_\Delta v_k$ are also isomorphic to $\Sigma$. In this way we determine $\Delta_{[3]}=\lk_\Delta v_i\cup \lk_\Delta v_j$: it is the union of three octahedral 2-spheres, each having a pair of antipodal facets $(F_1, F_2)$, $(F_2, F_3)$ and $(F_3, F_1)$, respectively. This also determines $\Delta$, which is the balanced triangulation of the nonorientable $\Sp^2$-bundle over $\Sp^1$ known as $^{3}12^{83}_2$ in \cite{KL}.
	\end{proof}
\begin{remark}
	The balanced 2-neighborly manifold $^{3}12^{83}_2$ is also known as $BM_4$ defined in \cite{KN}. In particular in \cite[Proposition 6.9]{KN} it is shown that $^{3}12^{83}_2$ is the only balanced 12-vertex 3-manifold with $\beta_1\neq 0$. See \cite{KN} and \cite{Z} for extension in higher dimensional cases.
\end{remark}
 
Next we characterize all balanced 3-spheres with each color sets of size 3.
	
		\begin{lemma}\label{lm: 3 triangulations of balanced 3-sphere}
			Up to an isomorphism, there are four triangulations of balanced $3$-spheres with each color set of size $3$.
		\end{lemma}
		\begin{proof}
			Let $\Delta$ be such a sphere and let $V_4=\{v_1,v_2,v_3\}$. Each vertex link of $\Delta$ is a balanced 2-sphere with at most 9 vertices, hence it is either the octahedral sphere, the suspension of a 6-cycle, or the connected sum of two octahedral spheres. We denote these three 2-spheres as $\Sigma_1$, $\Sigma_2$ and $\Sigma_3$ respectively. By Lemma \ref{lm: intersection of links is sphere}, $\Delta_{[3]}$ is the union of three triangulated 2-balls $B_i=\lk_\Delta v_j\cap \lk_\Delta v_k$, where $\{i,j,k\}=[3]$, glued along their common boundary complex $c$. Assume that $f_0(\lk_\Delta v_i)\leq f_0(\lk_\Delta v_j)$ when $i\leq j$. An easy counting leads to
			\begin{equation}\label{cond}
				f_0(\Delta_{[3]})=f_0(c)+\sum_{i=1}^{3} f_0(B_i\backslash c)=9, \quad f_0(\lk_\Delta v_i)=f_0(c)+f_0(B_j\backslash c)+f_0(B_k\backslash c)\in \{6,8,9\},
			\end{equation}
		    where $f_0(B_i\backslash c)$ counts the number of interior vertices of $B_i$. By the Dehn-Sommerville relations, the $f$-vector of any triangulated 3-manifold satisfies that $f_1=f_3+f_0$. Since every facet of $\Delta$ contains exactly one vertex from $V_4$, we have that $f_3(\Delta)=\sum_{i=1}^3 f_2(\lk_\Delta v_i)$ and hence $$f_1(\Delta)=\sum_{i=1}^3 f_2(\lk_\Delta v_i)+f_0(\Delta)=\sum_{i=1}^{3}(2f_0(\lk_\Delta v_i)-4)+12=2\sum_{i=1}^{3}f_0(\lk_\Delta v_i)\leq 54,$$we enumerate the combinatorial type of each $f_0(\lk_\Delta v_i)$ as follows:
		
		  {\bf Case 1:} $\lk_\Delta v_1 \cong \Sigma_1$. It follows that $\lk_\Delta v_3$ is obtained from $\lk_\Delta v_2$ by a cross flip (see \cite{IKN} for a reference). Since $f_0(\Delta_{[3]})=9$, either $\lk_\Delta v_2\cong \Sigma_1$, $\lk_\Delta v_3\cong \Sigma_3$, and the cross flip replaces a 2-face of $\lk_\Delta v_2$ with its complement in the octahedral sphere. Or $\lk_\Delta v_2\cong \Sigma_2$, $\lk_\Delta v_3\cong \Sigma_3$, and the cross flip replaces the union of three 2-faces of $\lk_\Delta v_2$ with its complement in the octahedral sphere. In the first case the 3-sphere is the connected sum of two octahedral 3-spheres, which we denote as $S_1$. In the second case we obtain a 3-sphere $S_2$ (with $\lk_{S_2}v_i\cong \Sigma_i$ for $i\in [3]$). Their $f$-vectors are $f(S_1)=(1, 12, 42, 60, 30)$ and $f(S_2)=(1, 12, 46, 68, 34)$.
			
		{\bf Case 2:} $\lk_\Delta v_1\cong \Sigma_2$. In this case the number of missing colored edges in $\Delta_{[3]}$ is
		$\frac{9\cdot 6}{2}-f_1(\Delta_{[3]})=27-\sum_{i=1}^{3}f_0(\lk_\Delta v_i)$, which equals either 1,2 or 3. 
		
		{\bf Subcase 1:} $\lk_\Delta v_2\cong \lk_\Delta v_3\cong \Sigma_2$. By conditions (\ref{cond}), $c$ is a 6-cycle and  $\Delta_{[3]}\backslash c$ consists of three disjoint vertices of degree either 4 or 6. Note that in $\Sigma_2$ every pair vertices of degree 4 and degree 6 forms an edge. Since $\Delta_{[3]}$ has only three missing edges between vertices of different colors, the vertices in $\Delta_{[3]}\backslash c$ are of the same color and has degree 6. Hence $\Delta_{[3]}$ is the join of $c$ and three disjoint vertices and $\Delta$ is the join of two 6-cycles. Denote this sphere as $S_3$; its $f$-vector is $(1, 12, 48, 72, 36)$. 
			   
	    {\bf Subcase 2:} $\lk_\Delta v_i\cong \Sigma_i$ for $i=2,3$. Then $c$ is a 7-cycle; furthermore, $B_3$ has no interior vertices and $B_1, B_2$ have a unique interior vertex $b_1, b_2$ respectively. Since $\Delta_{[3]}$ has two missing colored edges, and three vertices of degree 6 form an empty triangle in $\lk_\Delta v_3\cong \Sigma_3$, WLOG assume that $\deg b_1=4$ and $\deg b_2=6$. The only vertex $b_3$ not connected to $b_2$ in $B_2$ must be the vertex of degree 6 in $\lk_\Delta v_1 \cong \Sigma_2$. Hence $B_3$ is the join of $b_3$ with the path of length 5, which is $c\backslash b_3$. But then $b_3$ is also the degree-6 vertex in $\lk_\Delta v_2\cong \Sigma_2$, and there is no way to triangulate $B_1$ such that it shares no common interior edge with $B_2$.
	    
	    {\bf Subcase 3:} $\lk_\Delta v_2\cong \lk_\Delta v_3\cong \Sigma_3$. Then $c$ is an 8-cycle and only $B_1$ has an interior vertex $a$. Also $\Delta_{[3]}$ has one missing colored edge, so $a$ is of degree 6. Three vertices of degree 6 in $\Sigma_3$ are of different colors, hence in $\lk_\Delta v_2$ and $\lk_\Delta v_3$ the other two vertices of degree 6 must be two pairs of antipodal vertices in $\lk_{\Delta_{[3]}} a$ (the other pair of antipodal vertices in $\lk_{\Delta_{[3]}} a$ is the missing edge in $\Delta_{[3]}$). In this way we recover $\lk_\Delta v_2=B_1\cup B_3, \lk_\Delta v_3=B_1\cup B_2$, where $B_1, B_2, B_3$ are the 2-balls shown in the figure below (from left to right); in particular, labels of the vertices represent the color and the blue (red resp.) edges form the missing triangle in $\lk_\Delta v_2$ ($\lk_\Delta v_3$ resp.). We call this 3-sphere $S_4$. Its $f$-vector is $(1, 12, 52, 80, 40)$.
	    
	    {\bf Case 3:} $\lk_\Delta v_1\cong \Sigma_3$. This is not possible by Proposition \ref{prop: 2-nbly 3-sphere 12 ver}.
	    \begin{figure}[h]
	    	\centering
	    	\begin{tikzpicture}
	    	\draw (1.5, 0.5)node[right]{3}--(0.5, 1.5)node[right]{1}--(-0.5, 1.5)node[left]{2}--(-1.5, 0.5)node[left]{3}--(-1.5, -0.5)node[left]{2}--(-0.5, -1.5)node[left]{1}--(0.5, -1.5)node[right]{3}--(1.5, -0.5)node[right]{2}--cycle;
	    	\draw [blue](-1.5, -0.5)--(1.5, 0.5);
	    	\draw [red] (-0.5, 1.5)--(0.5, -1.5);
	    	\draw (-0.5, 1.5)--(1.5, 0.5);
	    	\draw (0.5, -1.5)--(-1.5, -0.5);
	    	\draw (-1.5, 0.5)--(0,0)node{1}--(1.5, -0.5);
	    	\end{tikzpicture}
	    	\hspace{5mm}
	    	\begin{tikzpicture}
	    	\draw (1.5, 0.5)node[right]{3}--(0.5, 1.5)node[right]{1}--(-0.5, 1.5)node[left]{2}--(-1.5, 0.5)node[left]{3}--(-1.5, -0.5)node[left]{2}--(-0.5, -1.5)node[left]{1}--(0.5, -1.5)node[right]{3}--(1.5, -0.5)node[right]{2}--cycle;
	    	\draw [red](-0.5, 1.5)--(0.5, -1.5);
	    	\draw (-1.5, 0.5)--(-0.5, -1.5)--(-0.5, 1.5);
	    	\draw (0.5, -1.5)--(0.5, 1.5)--(1.5, -0.5);
	    	\end{tikzpicture}
	    	\hspace{5mm}
	    	\begin{tikzpicture}
	    		\draw (1.5, 0.5)node[right]{3}--(0.5, 1.5)node[right]{1}--(-0.5, 1.5)node[left]{2}--(-1.5, 0.5)node[left]{3}--(-1.5, -0.5)node[left]{2}--(-0.5, -1.5)node[left]{1}--(0.5, -1.5)node[right]{3}--(1.5, -0.5)node[right]{2}--cycle;
	    		\draw [blue] (-1.5, -0.5)--(1.5, 0.5);
	    	\draw (-1.5, 0.5)--(0.5, 1.5)--(-1.5, -0.5);
	    	\draw (1.5, 0.5)--(-0.5, -1.5)--(1.5, -0.5);
	    	\end{tikzpicture}

        \end{figure}
    	\end{proof}
		\begin{theorem}\label{prop: 2-nbly 4-sphere 15 ver}
			There exists a unique balanced 2-neighborly homology $4$-sphere $^{4} 15 ^{5}_2$ with each color set of size 3.
		\end{theorem}
		\begin{proof}
			Let $\Delta$ be such a sphere and let its color set $V_5=\{v_1,v_2,v_3\}$. By Alexander Duality, $\tilde{H_i}(\Delta_{\{4,5\}})\cong \tilde{H}_{3-i}(\Delta_{[3]})$.
			In particular, since $\Delta_{\{4,5\}}$ is balanced 2-neighborly, $\beta_2(\Delta_{[3]})=\beta_1(\Delta_{\{4,5\}})=4$ and $\beta_1(\Delta_{[3]})=0$. Hence \[f_2(\Delta_{[3]})=(f_1-f_0+\chi)(\Delta_{[3]})=\frac{9\cdot 6}{2}-9+5=23.\]
			By double counting, $\sum_{i=1}^{3} f_1(\lk_\Delta v_i)=\sum_{W=\{i,j,5\}\subseteq [5]} f_2(\Delta_{W})=\binom{4}{2}f_2(\Delta_{[3]})=138$. By Proposition \ref{prop: 2-nbly 3-sphere 12 ver} and Lemma \ref{lm: 3 triangulations of balanced 3-sphere}, $f_1(\lk_\Delta v_i)\in \{42, 46, 48, 52\}$, it follows that either $138=42+48\cdot2$, that is, $\lk_\Delta v_1\cong S_1$ and $\lk_\Delta v_2, \lk_\Delta v_3\cong S_3$; or $138=46\cdot3$ and $\lk_\Delta v_i\cong S_2$ for all $i$.
			
			Consider the first case above. $S_1$ is the connected sum of two octahedral 3-spheres. For any 2-subset $W\subset [4]$, the induced subcomplex $(S_1)_W$ is the union of two 4-cycles glued along an edge, so $f_1((S_1)_W)=7$. Similarly, $S_3$ is the join of two 6-cycles, so we have $(S_3)_W$ is either a 6-cycle or the bipartite graph $K_{3,3}$, i.e., $f_1((S_3)_W)=6$ or 9. Hence $$23=f_2(\Delta_{W\cup\{5\}})=\sum_{i=1}^{3}f_1((\lk_\Delta v_i)_W)\in \{19, 22, 25\},$$ a contradiction.
			
			Now we consider the second case, where all vertex links in $\Delta$ are isomorphic to $S_2$. Let $\Gamma=\lk_\Delta v_1\cap \lk_\Delta v_2\cap \lk_\Delta v_3$. The proof of Lemma \ref{lm: 3 triangulations of balanced 3-sphere} implies that for any vertex $p\notin V_5$, $\lk_\Gamma p=\lk_{\lk_\Delta p} v_1\cap \lk_{\lk_\Delta p} v_2\cap \lk_{\lk_\Delta p} v_3$ is a 5-cycle (as the boundary of the union of three 2-faces, where we apply the cross-flip). Hence $\Gamma$ must be the boundary of the icosahedron. Since all $\lk_\Delta v_i$ are isomorphic,  by Lemma \ref{lm: intersection of links is sphere} $\Gamma$ divides the $3$-sphere $\lk_\Delta v_1$ into two 3-balls, each having the same number of facets. If the facets of $\lk_\Delta v_1$ are labeled as in the link of vertex 1 in $^4 {15}^5_2$ (this is a vertex-transitive triangulation of 4-sphere whose vertex links are isomorphic to $S_2$, see \cite{KL}), then one such $\Gamma$ is the intersection of vertices $1,6,8$ in $^4 15^5_2$; we rename it to $\Gamma_1$. In this case, $\lk_\Delta v_1=B\cup_{\Gamma_1} B'$, where $B, B'$ are isomorphic 3-balls. We check by sage \cite{Sage} that all other subcomplexes in $\lk_\Delta v_1$ that are isomorphic to $\Gamma_1$ are of the form $\sigma(\Gamma_1)$, where $\sigma$ is an element in the permutation group of $\lk_\Delta v_1$ (of order 8). So it suffices to consider just $\Gamma_1$. To reconstruct $\lk_\Delta v_2$ and hence $\Delta$, note that $\lk_\Delta u_2$ has the decomposition $\lk_\Delta u_2=B'\cup_{\Gamma_1} B''\cong S_2$ for some 3-ball $B''$, and furthermore $B''\cong B'$. To decide $B''$ it is equivalent to finding a balanced simplicial isomorphism $f: B'\to B''$ with $B'\cap B''=B\cap B''=\Gamma_1$ and $f(\Gamma_1)=\Gamma_1$; in other words, $f$ is a permutation in $\mathrm{Aut}(\Gamma_1)$. We check by sage \cite{Sage} that the links of vertex 6,8 in $^415^5_2$ are the only candidates for $\lk_\Delta v_2$. Hence $\Delta={^4}{15}^5_2$. Indeed $\Delta$ is balanced: the color sets are $\{1,6,8\}$, $\{2,4,9\}$, $\{3,7,11\}$, $\{5,10,15\}$ and $\{12,13,14\}$.
		\end{proof}

		\begin{theorem}\label{cor: 3-nbly 5-sphere 18 ver}
			There exists a balanced $3$-neighborly non-spherical $5$-manifold with each color set of size 3.
		\end{theorem}
		\begin{proof}
			By Theorem \ref{prop: 2-nbly 4-sphere 15 ver}, if such 5-manifold exists, then all vertex links are isomorphic to $^4 15^5_2$, which we denote as $\Gamma$. Based on the list of facets of $\Gamma$ in \cite{manifold}, we take a color-preserving permutation $\sigma=(1,6,8)(2,4,9)(11,3,7)(10,15,5)(13,14,12)$. We choose $\sigma$ in such a way that $\sigma\notin \mathrm
			{Aut}(\Gamma)$ and furthermore, $\Gamma\cap \sigma(\Gamma)$, $\Gamma\cap \sigma^2(\Gamma)$ and $\sigma(\Gamma)\cap \sigma^2(\Gamma)$ are isomorphic homology manifolds with no interior faces of dimension $<2$ and with a common boundary $C$. By computer we check that $\Gamma\cup \sigma(\Gamma)\cup \sigma^2(\Gamma)$ is balanced 3-neighborly and $C$ is the vertex-transitive 3-manifold $^3 15^{15}_1$ that triangulates $\Sp^3/Q$ as in \cite{manifold}. Finally let $\Delta =\left( \Gamma * \{16\} \right) \cup \left( \sigma(\Gamma)*\{17\}\right) \cup \left( \sigma^2(\Gamma)*\{18\}\right) $, where $\{16,17,18\}$ are the vertices of color 6. By sage \cite{Sage} one verifies that all vertex links of $\Delta$ is isomorphic to $\Gamma$, which is known as a combinatorial 4-sphere (we say a simplicial complex is a combinatorial sphere if it is PL homeomorphic to the boundary of the simplex). Hence $\Delta$ is a combinatorial manifold that is balanced 3-neighborly.
		\end{proof}
	\begin{remark}
		The following properties of the balanced 3-neighborly 5-manifold found in the proof of Theorem \ref{cor: 3-nbly 5-sphere 18 ver} are verified by sage:
		\begin{enumerate}
			\item It is vertex-transitive and has the following generators of the automorphism group (of order 1080):
			\[(2,15)(4,5)(9,10)(12,17)(13,18)(14,16),\;\; (1,2)(3,15)(4,6)(5,7)(8,9)(10,11),\] \[(1,3)(6,7)(8,11)(12,17)(13,18)(14,16),\;\;(1,6,8)(2,3,17)(4,7,18)(5,15,10)(9,11,16),\] \[(1,13)(3,18)(6,14)(7,16)(8,12)(11,17).\]
			\item The homology groups of $\Delta$ are given by $(\Z, 0, \Z_2, 0,0,\Z)$.
			\item The $f$-vector of $\Delta$ is $(1, 18, 135, 540, 1035, 918, 306)$.			
		\end{enumerate}
	
		Furthermore by the Dehn-Sommerville relations, any balanced 3-neighborly 5-manifold having 3 vertex in each color set also has the $f$-vector $(1,18,135, 540,1035,918,306)$. Let $\Sigma$ be such a complex, $\{v_1, v_2, v_3\}$ a color set, $M_i=\lk_\Sigma v_j\cap \lk_\Sigma v_k$ for $\{i,j,k\}=[3]$ and $N=M_1\cap M_2\cap M_3$. If $F$ is an interior face of $M_i$, then $F\cup \{v_i\}\notin \Sigma$. Since $\Sigma$ is balanced 3-neighborly, $\lk_\Sigma v_1=M_2\cup M_3\cong M_1\cup M_3=\lk_\Sigma v_2$ and they have 102 facets, it follows that each $M_i$ has no interior vertices or edges and with 51 facets. Also since all vertex links in $\Sigma$ are isomorphic to $^4 {15}^5_2$, the same argument as in the proof of Theorem \ref{prop: 2-nbly 4-sphere 15 ver} implies that $N$ is a 3-manifold whose vertex links are isomorphic to the boundary of the icosahedron; indeed $^3{15}^{15}_1$ is one such example. We haven't checked if there exist other balanced 3-neighborly 5-manifolds. (It is not known if there exist 15-vertex non-vertex-transitive 3-manifolds whose vertex links are all isomorphic to the boundary of the icosahedron.)
	\end{remark}
	
	\section{Balanced $2$-neighborly $3$-sphere with $16$ vertices}
In this section we provide a balanced 2-neighborly triangulation of the 3-sphere. The construction is motivated by Lemma \ref{lm: intersection of links is sphere}.
\begin{construction}\label{Second Example}
	\begin{figure}[h]
		\centering
		\begin{tikzpicture}
		\newdimen\R
		\R=1.9cm
		\draw (0:\R)
		-- cycle (360:\R) node[right] {$w_3$}
		-- cycle (30:\R) node[above right] {$u_3$}
		-- cycle (60:\R) node[above right] {$v_3$}
		-- cycle (90:\R) node[above] {$w_1$}
		-- cycle (120:\R) node[above left] {$u_2$}
		-- cycle (150:\R) node[above left] {$v_1$}
		-- cycle (180:\R) node[left] {$u_1$}
		-- cycle (210:\R) node[below left] {$w_2$}
		-- cycle (240:\R) node[below left] {$v_4$}
		-- cycle (270:\R) node[below] {$u_4$}
		-- cycle (300:\R) node[below right] {$w_4$}
		-- cycle (330:\R) node[below right] {$v_2$};
		\draw[thick,red, fill=red!5] (330:\R) -- (30:\R) -- (90:\R)--cycle;
		\draw[thick,red] (330:\R) -- (120:\R) -- (300:\R)--(150:\R);
		\draw[thick,red, fill=red!5] (150:\R) -- (270:\R) -- (210:\R)--cycle;
		\draw [thick,blue] (0:\R)--(30:\R) --(60:\R)-- (90:\R)--(120:\R)--(150:\R)--(180:\R)--(210:\R) --(240:\R)--(270:\R)--(300:\R)--(330:\R)--(0:\R);
		\end{tikzpicture}
		\hspace{2mm}
		\begin{tikzpicture}
		\newdimen\R
		\R=1.9cm
		\draw (0:\R)
		-- cycle (360:\R) node[right] {$w_3$}
		-- cycle (30:\R) node[above right] {$u_3$}
		-- cycle (60:\R) node[above right] {$v_3$}
		-- cycle (90:\R) node[above] {$w_1$}
		-- cycle (120:\R) node[above left] {$u_2$}
		-- cycle (150:\R) node[above left] {$v_1$}
		-- cycle (180:\R) node[left] {$u_1$}
		-- cycle (210:\R) node[below left] {$w_2$}
		-- cycle (240:\R) node[below left] {$v_4$}
		-- cycle (270:\R) node[below] {$u_4$}
		-- cycle (300:\R) node[below right] {$w_4$}
		-- cycle (330:\R) node[below right] {$v_2$};
		\draw[thick,red] (0:\R) -- (60:\R) -- (120:\R)--cycle;
		\draw[thick,red] (150:\R) -- (0:\R) -- (180:\R)--(330:\R);
		\draw[thick,red] (180:\R) -- (300:\R) -- (240:\R)--cycle;
		\draw [thick,blue] (0:\R)--(30:\R) --(60:\R)-- (90:\R)--(120:\R)--(150:\R)--(180:\R)--(210:\R) --(240:\R)--(270:\R)--(300:\R)--(330:\R)--(0:\R);
		\end{tikzpicture}
		\hspace{2mm}
		\begin{tikzpicture}
		\newdimen\R
		\R=1.9cm
		\draw (0:\R)
		-- cycle (360:\R) node[right] {$w_3$}
		-- cycle (30:\R) node[above right] {$u_3$}
		-- cycle (60:\R) node[above right] {$v_3$}
		-- cycle (90:\R) node[above] {$w_1$}
		-- cycle (120:\R) node[above left] {$u_2$}
		-- cycle (150:\R) node[above left] {$v_1$}
		-- cycle (180:\R) node[left] {$u_1$}
		-- cycle (210:\R) node[below left] {$w_2$}
		-- cycle (240:\R) node[below left] {$v_4$}
		-- cycle (270:\R) node[below] {$u_4$}
		-- cycle (300:\R) node[below right] {$w_4$}
		-- cycle (330:\R) node[below right] {$v_2$};
		\draw[thick,red] (150:\R) -- (90:\R) -- (180:\R)--(60:\R) -- (210:\R) -- (30:\R)--(240:\R)--(0:\R) -- (270:\R) -- (330:\R);
		\draw [thick,blue] (0:\R)--(30:\R) --(60:\R)-- (90:\R)--(120:\R)--(150:\R)--(180:\R)--(210:\R) --(240:\R)--(270:\R)--(300:\R)--(330:\R)--(0:\R);
		\end{tikzpicture}
		\caption{Discs $A$, $B$ and $C$ (from left to right)}\label{fig: patches A,B,C}
	\end{figure}
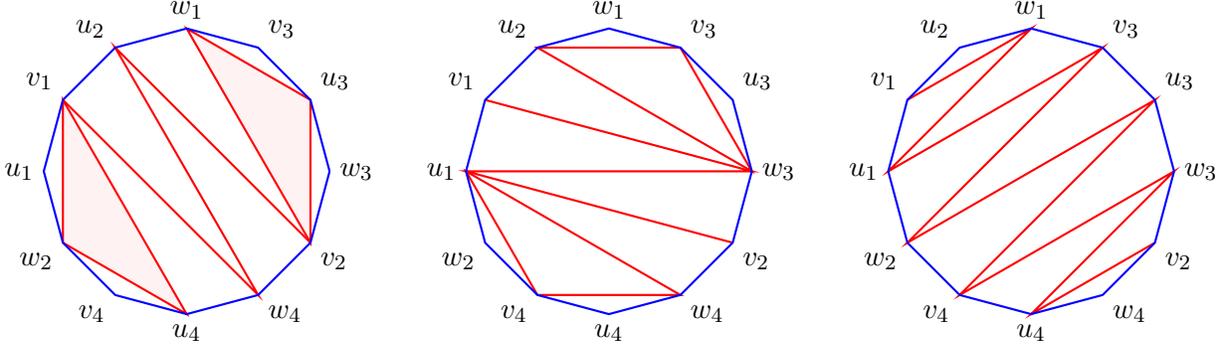
	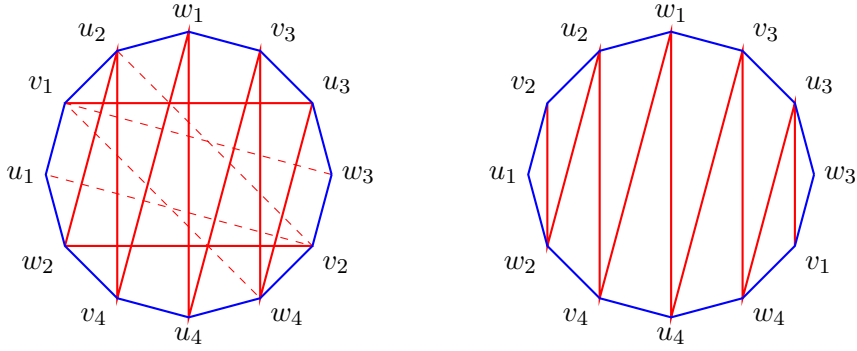
\begin{figure}[h]
		\centering
		\begin{tikzpicture}
		\newdimen\R
		\R=1.9cm
		\draw (0:\R)
		-- cycle (360:\R) node[right] {$w_3$}
		-- cycle (30:\R) node[above right] {$u_3$}
		-- cycle (60:\R) node[above right] {$v_3$}
		-- cycle (90:\R) node[above] {$w_1$}
		-- cycle (120:\R) node[above left] {$u_2$}
		-- cycle (150:\R) node[above left] {$v_1$}
		-- cycle (180:\R) node[left] {$u_1$}
		-- cycle (210:\R) node[below left] {$w_2$}
		-- cycle (240:\R) node[below left] {$v_4$}
		-- cycle (270:\R) node[below] {$u_4$}
		-- cycle (300:\R) node[below right] {$w_4$}
		-- cycle (330:\R) node[below right] {$v_2$};
		\draw[thick,red] (210:\R) -- (120:\R) -- (240:\R)--(90:\R) -- (270:\R) -- (60:\R)--(300:\R)-- (30:\R);
		\draw[thick, red] (150:\R)--(30:\R);
		\draw[thick, red] (210:\R)--(330:\R);
		\draw[dashed,red] (0:\R)--(150:\R)--(300:\R);
		\draw[dashed,red] (120:\R)--(330:\R)--(180:\R);
		\draw [thick,blue] (0:\R)--(30:\R) --(60:\R)-- (90:\R)--(120:\R)--(150:\R)--(180:\R)--(210:\R) --(240:\R)--(270:\R)--(300:\R)--(330:\R)--(0:\R);
		\end{tikzpicture}
		\hspace{10mm}
		\begin{tikzpicture}
		\newdimen\R
		\R=1.9cm
		\draw (0:\R)
		-- cycle (360:\R) node[right] {$w_3$}
		-- cycle (30:\R) node[above right] {$u_3$}
		-- cycle (60:\R) node[above right] {$v_3$}
		-- cycle (90:\R) node[above] {$w_1$}
		-- cycle (120:\R) node[above left] {$u_2$}
		-- cycle (150:\R) node[above left] {$v_2$}
		-- cycle (180:\R) node[left] {$u_1$}
		-- cycle (210:\R) node[below left] {$w_2$}
		-- cycle (240:\R) node[below left] {$v_4$}
		-- cycle (270:\R) node[below] {$u_4$}
		-- cycle (300:\R) node[below right] {$w_4$}
		-- cycle (330:\R) node[below right] {$v_1$};
		\draw[thick,red] (150:\R)--(210:\R) -- (120:\R) -- (240:\R)--(90:\R) -- (270:\R) -- (60:\R)--(300:\R)-- (30:\R)--(330:\R);
		\draw [thick,blue] (0:\R)--(30:\R) --(60:\R)-- (90:\R)--(120:\R)--(150:\R)--(180:\R)--(210:\R) --(240:\R)--(270:\R)--(300:\R)--(330:\R)--(0:\R);
		\end{tikzpicture}
		\caption{Left: disc $D'$. Right: disc $D$ obtained after rearranging the boundary of $D'$.}\label{fig: patch D} 
	\end{figure}
	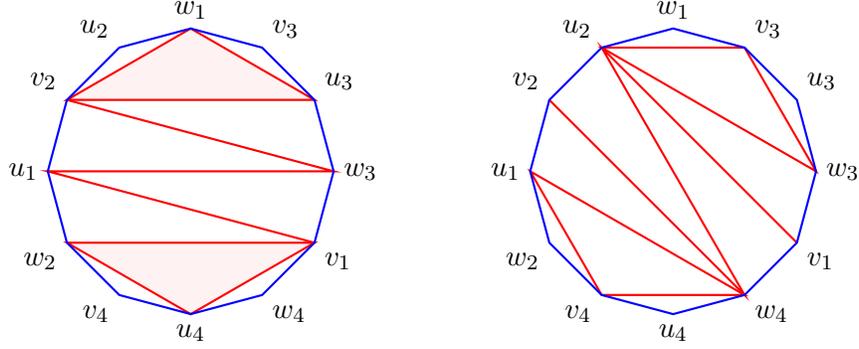
\begin{figure}[h]
		\centering
		\begin{tikzpicture}
		\newdimen\R
		\R=1.9cm
		\draw (0:\R)
		-- cycle (360:\R) node[right] {$w_3$}
		-- cycle (30:\R) node[above right] {$u_3$}
		-- cycle (60:\R) node[above right] {$v_3$}
		-- cycle (90:\R) node[above] {$w_1$}
		-- cycle (120:\R) node[above left] {$u_2$}
		-- cycle (150:\R) node[above left] {$v_2$}
		-- cycle (180:\R) node[left] {$u_1$}
		-- cycle (210:\R) node[below left] {$w_2$}
		-- cycle (240:\R) node[below left] {$v_4$}
		-- cycle (270:\R) node[below] {$u_4$}
		-- cycle (300:\R) node[below right] {$w_4$}
		-- cycle (330:\R) node[below right] {$v_1$};
		\draw [thick, red, fill=red!5] (30:\R)--(90:\R)--(150:\R)--cycle;
		\draw [thick, red,fill=red!5] (210:\R)--(270:\R)--(330:\R)--cycle;
		\draw [thick, red] (150:\R)--(0:\R)--(180:\R)--(330:\R);
		\draw [thick,blue] (0:\R)--(30:\R) --(60:\R)-- (90:\R)--(120:\R)--(150:\R)--(180:\R)--(210:\R) --(240:\R)--(270:\R)--(300:\R)--(330:\R)--(0:\R);
		\end{tikzpicture}
		\hspace{10mm}
		\begin{tikzpicture}
		\newdimen\R
		\R=1.9cm
		\draw (0:\R)
		-- cycle (360:\R) node[right] {$w_3$}
		-- cycle (30:\R) node[above right] {$u_3$}
		-- cycle (60:\R) node[above right] {$v_3$}
		-- cycle (90:\R) node[above] {$w_1$}
		-- cycle (120:\R) node[above left] {$u_2$}
		-- cycle (150:\R) node[above left] {$v_2$}
		-- cycle (180:\R) node[left] {$u_1$}
		-- cycle (210:\R) node[below left] {$w_2$}
		-- cycle (240:\R) node[below left] {$v_4$}
		-- cycle (270:\R) node[below] {$u_4$}
		-- cycle (300:\R) node[below right] {$w_4$}
		-- cycle (330:\R) node[below right] {$v_1$};
		\draw [thick, red] (0:\R)--(60:\R)--(120:\R)--cycle;
		\draw [thick, red] (180:\R)--(240:\R)--(300:\R)--cycle;
		\draw [thick, red] (330:\R)--(120:\R)--(300:\R)--(150:\R);
		\draw [thick,blue] (0:\R)--(30:\R) --(60:\R)-- (90:\R)--(120:\R)--(150:\R)--(180:\R)--(210:\R) --(240:\R)--(270:\R)--(300:\R)--(330:\R)--(0:\R);
		\end{tikzpicture}
		\caption{Left: disc $A'$. Right: disc $B'$. Notice that $\partial A'=\partial B'=\partial D$.}
		\label{figure: A' and B'}
	\end{figure}
	Assume that $V_1=\{u_1, u_2, u_3, u_4\}$, $V_2=\{v_1,v_2,v_3,v_4\}$, $V_3=\{w_1,w_2,w_3,w_4\}$ and $V_4=\{z_1,z_2,z_3,z_4\}$ are the four color sets of a balanced 3-sphere $\Gamma$. We let $\lk_\Gamma z_1=A\cup_{\partial A \sim \partial C} C$ and $\lk_\Gamma z_3=B\cup_{\partial B \sim \partial C} C$, where $A$, $B$ and $C$ are triangulated 2-balls sharing the same boundary as shown in Figure \ref{fig: patches A,B,C}. All possible edges that do not appear in $A$, $B$ and $C$ are shown in Figure \ref{fig: patch D} as solid red edges in disc $D'$. Notice that the dashed edges in $D'$ are edges in discs $A$ and $B$, so we may rearrange the boundary of $D$ by switching the positions of vertices $v_1$ and $v_2$, and then replacing the edges containing $v_1$ or $v_2$ in $\partial D'$ by the dashed edges. In this way, we obtain a triangulation of a 12-gon $D$ as shown in Figure \ref{fig: patch D}. Furthermore, $\partial D\subseteq A\cup B$, and $\partial D$ divides the sphere $=A\cup_{\partial A \sim \partial B} B$ into two discs $A'$ and $B'$ as shown in Figure \ref{figure: A' and B'}. 	
	
	We let $\lk_\Gamma z_2=A'\cup_{\partial A' \sim \partial D} D$ and $\lk_\Gamma z_4=B'\cup_{\partial B' \sim \partial D} D$. Since both $\st_\Gamma z_1\cap \st_\Gamma z_3=C$ and $\st_\Gamma z_2\cap (\st_\Gamma z_1\cup \st_\Gamma z_3)=A'$ are simplicial 2-balls, it follows that $\Sigma=\cup_{i=1}^{3} \st_\Gamma z_i$ is a simplicial 3-ball. Furthermore, the boundary of $\Sigma$ is exactly $\lk_\Gamma z_4$. Hence $\Gamma=\Sigma\cup\st_\Gamma z_4$ is indeed a balanced 2-neighborly 3-sphere. 
\end{construction}

\begin{remark} Here we provide some properties of $\Gamma$ in Construction \ref{Second Example}. 
	\begin{enumerate}
		\item $(A\cup B, C, D)$ is an ear decomposition of $\Gamma_{[3]}$.
		
		\item The automorphism group of $\Gamma$ has two generators \[(u_1u_3u_2u_4)(v_1z_2v_2z_1)(v_3z_4v_4z_3)(w_1w_4w_2w_3),\: (z_1v_1)(z_2v_2)(z_3v_3)(z_4v_4)(u_1w_1)(u_2w_2)(u_3w_3)(u_4w_4).\] (The second generator is given by switching vertices of color 1 and 3, and color 2 and 4, but with the same subscript.) Hence ${\rm{Aut}}(\Gamma)$ has 8 elements.
		
		\item The complex $\Gamma$ given in Construction \ref{Second Example} is shellable. For $\lk_\Gamma z_1=A\cup_{\partial A\sim \partial C}C$, there exist two shellings $c_1,\ldots,c_{10}, a_1,\ldots, a_{10}$ and $a'_1,\ldots,a'_{10}, c'_1,\ldots,c'_{10}$ such that for any $1\leq i\leq 10$, $c_i, c'_i$ are facets from $C$ and $a_i, a'_i$ are facets from $A$. Similarly, there exist two shellings $c_1,\ldots,c_{10}, b_1,\ldots, b_{10}$ and $b'_1,\ldots,b'_{10},c'_1,\ldots,c'_{10}$  for $\lk_\Gamma z_3=B\cup_{\partial B\sim \partial C}C$, where $b_i,b'_i$ are facets from $B$. Then
		\[a'_1*z_1,\ldots,a'_{10}*z_1, c'_1*z_1,\ldots,c'_{10}*{z_1}, c_1*z_3, \ldots, c_{10}*z_3,b_1*z_3,\ldots,b_{10}*z_3\]
		gives a shelling of $\st_\Gamma z_1\cup \st_\Gamma z_3$. We may extend this shelling into a complete shelling of $\Gamma$ by constructing two similar shellings of $\lk_\Gamma z_2$ and $\lk_\Gamma z_4$. However, we tried some computer tests and failed to prove either polytopality or non-polytopality. 
	\end{enumerate}
	
\end{remark}
\begin{remark}
	It is easy to see that if $\Delta_1$ is a balanced 2-neighborly $(d_1-1)$-sphere and $\Delta_2$ is a balanced 2-neighborly $(d_2-1)$-sphere, then $\Delta_1*\Delta_2$ is a balanced 2-neighborly $(d_1+d_2-1)$-sphere. Hence by taking joins, we find balanced 2-neighborly $(4k-1)$-spheres with $16k$ vertices for any $k\geq 1$.
\end{remark}
\begin{question}
	Let $d\geq 4$ and $m\geq 5$ be arbitrary integers. Is there a balanced 2-neighborly simplicial $(d-1)$-sphere all of whose color sets have the same size $m$? Is there a polytopal sphere with these properties?
\end{question}

	\section{Balanced $2$-neighborly $L(3,1)$ with $16$ vertices}
		In this section we present our first construction of a balanced 2-neighborly lens space $L(3,1)$ with 16 vertices. We denote it by $\Delta$. Each color set of $\Delta$ has four vertices.
		
		\begin{figure}[h]
			\centering
			\subfloat[$\lk_\Delta z_1$]{\includegraphics{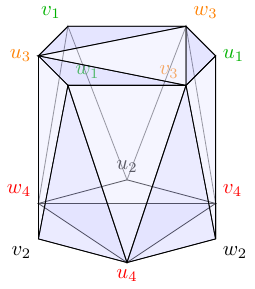}}
			\hspace{15mm}
			\subfloat[$\lk_\Delta z_2$]{\includegraphics{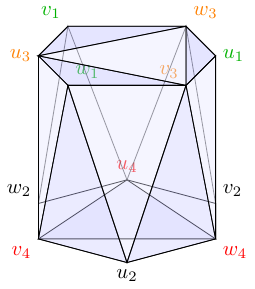}}
			\\
			\subfloat[$\lk_\Delta z_3$]{\includegraphics{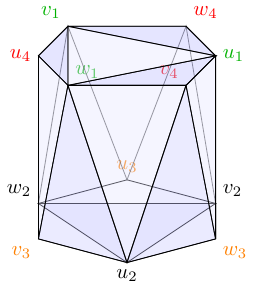}}
			\hspace{15mm}	
			\subfloat[$\lk_\Delta z_4$]{\includegraphics{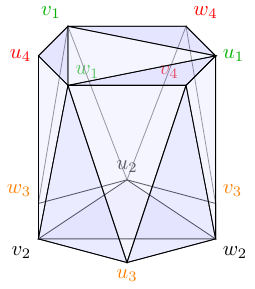}}
			\caption{Four vertex links of $\Delta$}	
			\label{Figure: four links}
		\end{figure}
		
		\begin{construction}\label{First Example}
			We denote the color sets of $\Delta$ by $V_1=\{u_1,u_2,u_3,u_4\}$, $V_2=\{v_1,v_2,v_3,v_4\}$, $V_3=\{w_1,w_2,w_3,w_4\}$ and $V_4=\{z_1,z_2,z_3,z_4\}$. 
			
			In Figure \ref{Figure: four links} we illustrate the construction of the vertex links $\lk_\Delta z_i$ for $i=1,\ldots,4$. All these links are realized as cylinders. Two links $\lk_\Delta z_1$ and $\lk_\Delta z_2$ share the same top and bottom, which are triangulated hexagons spanned by vertices $\{u_i,v_i,w_i:i=1,3\}$ and $\{u_i,v_i,w_i:i=2,4\}$, respectively. To construct $\lk_\Delta z_3$ from $\lk_\Delta z_1$, we switch the positions of vertices $u_3, v_3, w_3$ with vertices $u_4,v_4,w_4$ respectively and form a new cylinder. The new top and bottom hexagons contain the 2-faces $\{u_1,v_1,w_1\}$ and $\{u_2,v_2,w_2\}$. Similarly, we construct the link $\lk_\Delta z_4$ from $\lk_\Delta z_2$ by switching the positions of vertices $u_3, v_3, w_3$ with vertices $u_4,v_4,w_4$ and letting $\{u_1,v_1,w_1\}$ and $\{u_2,v_2,w_2\}$ be the 2-faces that appear in the triangulation of the top and bottom hexagons. It follows that $\lk_\Delta z_3$ and $\lk_\Delta z_4$ also share the same top and bottom.
			
			Now since $\Delta$ is balanced 2-neighborly, by our construction, it only remains to show that $\Delta$ triangulates the lens space $L(3,1)$. The geometric realizations of $\st_\Delta z_1$ and $\st_\Delta z_2$ are filled cylinders that share top and bottom. So their union $A:=\st_\Delta z_1\cup \st_\Delta z_2$ is a filled torus (that is, a genus-1 handlebody); so is the union $B:=\st_\Delta z_3\cup \st_\Delta z_4$. Note that these two handlebodies have identical boundary complexes, thus they provide a Heegaard splitting of a lens space.
			
			To identify which lens space $\Delta$ triangulates, we need to determine the homeomorphism $\phi: \partial A\to \partial B$. Consider two generators $\gamma, \delta$ of $\pi_1(A\cap B)= \pi_1(\partial A)$, where $\gamma$ is the 6-cycle $(u_3,v_1,w_3,u_1,v_3,w_1)$ and $\delta$ is the 4-cycle $(u_1,w_2,u_4,w_3)$. In particular, $\delta$ is also a generator of $\pi_1(A)$. From the construction we see that $\phi(\gamma)$ is a loop running around the equator of $\partial B$ thrice and the meridian of $\partial B$ once. Also $\phi(\delta)$ runs around the equator of $\partial B$ twice and the meridian of $\partial B$ once. Hence it is indeed the lens space $L(3,1)$.
			
			\begin{remark}\label{rm: property of construction 1}
			Our construction $\Delta$ has the following properties:
			\begin{enumerate}
				\item All vertex links are combinatorially equivalent. 
				\item From Figure 5 we see $\lk_\Delta z_i\cap\lk_\Delta z_j$ has two connected components when $\{i,j\}=\{1,2\}$ or $\{3,4\}$ (they are the top and bottom hexagons as shown in Figure 2); and it has three connected components when $i\in\{1,2\}$ and $j\in\{3,4\}$ (each component is the union of two facets along the side of the cylinders). In general, the intersection of two vertex links, where the vertices are of the same color, always has at least two connected components.
					\item There are three group actions on the vertices of $\Delta$:
					\begin{enumerate}
						\item Fix the subscript and rotate the corresponding vertices of color 1, 2 and 3 respectively. The generator is given by $(u_1v_1w_1)(u_2v_2w_2)(u_3v_3w_3)$.
						\item Rotate vertices of the same color. The generator is \[(u_1u_3u_2u_4)(v_1v_3v_2v_4)(w_1w_3w_2w_4)(z_1z_3z_2z_4).\]
						\item Exchange $\lk_\Delta z_1$ and $\lk_\Delta z_2$, $\lk_\Delta z_3$ and $\lk_\Delta z_4$, by exchanging $v_i$ and $w_i$ (or $u_i$ and $w_i$, $u_i$ and $v_i$) for all $i\in [4]$. The generators are $(z_1z_2)(z_3z_4)(v_1w_1)(v_2w_2)(v_3w_3)(v_4w_4)$, $(z_1z_2)(z_3z_4)(u_1w_1)(u_2w_2)(u_3w_3)(u_4w_4)$ and $(z_1z_2)(z_3z_4)(u_1v_1)(u_2v_2)(u_3v_3)(u_4v_4)$.
					\end{enumerate}
					The automorphism group of $\Delta$ is of size 96.  
				\end{enumerate}
			\end{remark}
		\end{construction}
		
		\begin{proposition}
			The complex $\Delta$ is a balanced vertex minimal triangulation of $L(3,1)$.
		\end{proposition}
		
		\begin{proof}
			By Proposition 6.1 in \cite{KN}, each color set of $\Delta$ is of size at least 3. If there are exactly three vertices $v_1, v_2, v_3$ of color 1 in $\Delta$, apply the Mayer-Vietoris sequence on the triple $(\st_\Delta v_1\cup\st_\Delta v_2, \st_\Delta v_3, \Delta)$ and we obtain that 
			\[0=H_1(\lk_\Delta v_3)\to H_1(\st_\Delta v_1\cup\st_\Delta v_2)\oplus H_1(\st_\Delta v_3) \to H_1(\Delta)\to H_0(\lk_\Delta v_3)=0.\]
			Hence $H_1(\st_\Delta v_1\cup\st_\Delta v_2)\cong H_1(\Delta)=\mathbb{Z}/3\mathbb{Z}$. However, this is impossible since $H_1(\st_\Delta v_1\cup\st_\Delta v_2)\cong H_0(\st_\Delta v_1\cap\st_\Delta v_2)$, which cannot be $\mathbb{Z}/3\mathbb{Z}$.
		\end{proof}
		
		The same argument as above also shows that the balanced triangulation of any lens space $L(p,q)$ with $p>1$ must have at least 16 vertices.

	\section{Balanced spheres and ear decomposition}
	In this section our goal is to construct a balanced 3-sphere whose rank-selected subcomplexes do not have ear decompositions. The motivation is from the balanced 2-neighborly construction of $L(3,1)$ in Section 5. Indeed, we want to construct a balanced 3-dimensional complex $\Delta$ so that 1) each vertex link is a 2-sphere; 2) for a fixed color set $V_4=\{v_1,\cdots, v_k\}$, the intersection of any two vertex links $\lk_\Delta v_i\cap \lk_\Delta v_j$ always has at least two connected components (as the property listed in Remark \ref{rm: property of construction 1}); and 3) $\cup_{i=1}^{4}\st_\Delta v_i$ is 3-ball, which together with the condition 1) guarantees that $\Delta$ is a 3-sphere.
	
	In the following we take $k=5$ and give such a construction. Figure \ref{fig1:links} illustrates the links $\lk_\Delta v_1,\cdots,\lk_\Delta v_4$. Every label represents the color of the vertex. Also each connected component of $\lk_\Delta v_1\cap \lk_\Delta v_2$ is colored in green, $\lk_\Delta v_i\cap \lk_\Delta v_3$ is colored in blue for $i=1,2$, and $\lk_\Delta v_j\cap \lk_\Delta v_4$ is colored in pink for $j=1,2,3$. Immediately we check that all these intersections of vertex links have 2 or 3 connected components.
	\begin{figure}[h]
			\centering
			\subfloat[$\lk_\Delta v_1$ and $\lk_\Delta v_2$]{\includegraphics[scale=0.6]{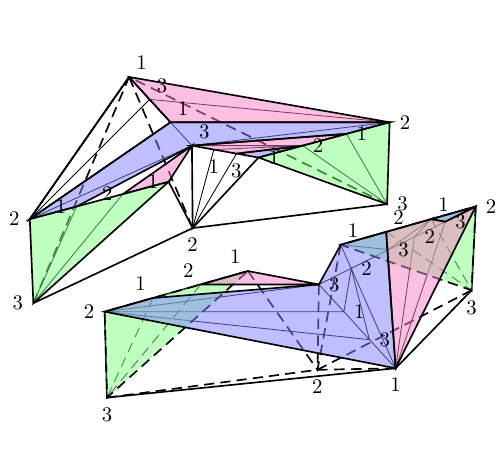}\label{link1,2}}
			\hspace{5mm}
			\subfloat[$\lk_\Delta v_3$]{\includegraphics[scale=0.6]{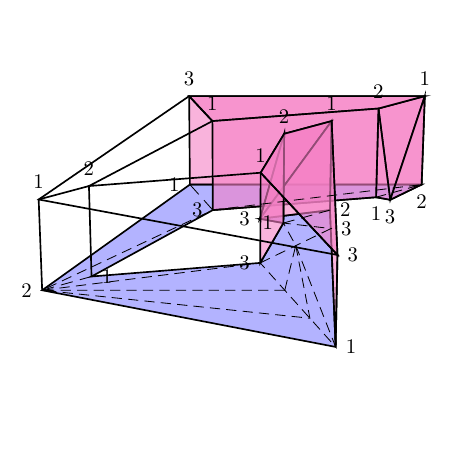}\label{link3}}
			\subfloat[$\lk_\Delta v_4$]{\includegraphics[scale=0.6]{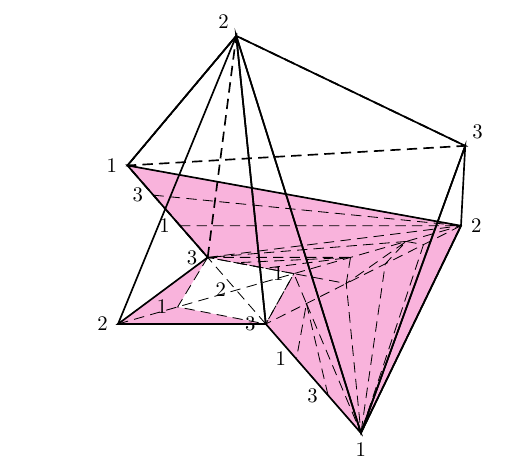}\label{link4}}
			\caption{Four vertex links as triangulated 2-spheres. For simplicity's sake, we omit some diagonal edges in the quadrilaterals in (b), and some labels of vertices in (c).}\label{fig1:links}	
		\end{figure}
		
		\begin{figure}[h]
			\centering
			\subfloat[$\lk_\Delta v_1\cup\lk_\Delta v_2$]{\includegraphics[scale=0.6]{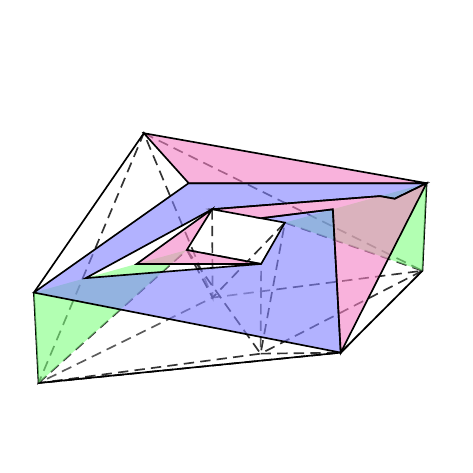}\label{link12}}
			\hspace{3mm}
			\subfloat[$\cup_{i=1}^{3}\lk_\Delta v_i$]{\includegraphics[scale=0.6]{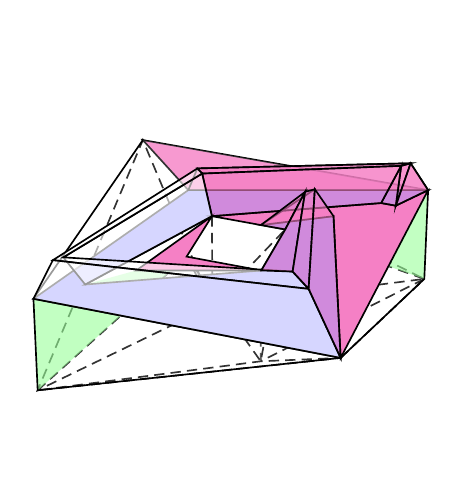}\label{link123}}
			\subfloat[$\cup_{i=1}^{4}\lk_\Delta v_4$]{\includegraphics[scale=0.6]{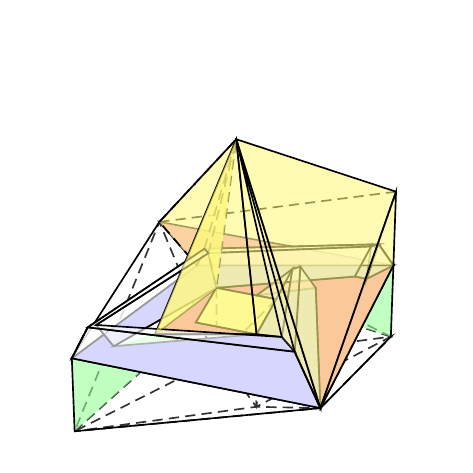}\label{link1234}}
			\caption[short]{how the links are glued together.}\label{fig2:union of links}	
		\end{figure}
		
		Figure \ref{fig2:union of links} shows how $\Delta\backslash V_4$ is formed from these links. First we glue $\lk_\Delta v_1$ and $\lk_\Delta v_2$ along two green triangles. The resulting complex $\lk_\Delta v_1\cup \lk_\Delta v_2$ is shown in Figure \ref{link12}. Then we place $\lk_\Delta v_3$ on top of $\lk_\Delta v_1\cup\lk_\Delta v_2$. As we see from Figure \ref{link123}, the boundary complex of $\cup_{i=1}^{3}\st_\Delta v_i$ is a triangulated torus. Finally, we place $\lk_\Delta v_4$ on top of $\cup_{i=1}^{3}\lk_\Delta v_i$ so that $\st_\Delta v_4$ ``covers the 1-dimensional hole" in $\cup_{i=1}^{3}\st_\Delta v_i$, see Figure \ref{link1234}. We denote the subspace of $\mathbb{R}^3$ enclosed by $\lk_\Delta v_i$ as $S_i$ for $1\leq i\leq 4$, and let $S_5:=\cup_{i\leq 4} S_i$. From our construction it follows that the boundary complex of $S_5$ is a 2-sphere; we let it be $\lk_\Delta v_5$. Indeed $\Delta$ is a 3-sphere since $\Delta$ is the union of two 3-balls $S_5$ and $\st_\Delta v_5$ glued along their common boundary $\lk_\Delta v_5$.
		
		Since each $\lk_\Delta v_i\cap \lk_\Delta v_j$ has at least two connected components for $1\leq i\neq j\leq 4$, the Mayer-Vietoris sequence implies that $S_i\cup S_j$ is not contractible for all $1\leq i\neq j\leq 4$. A similar inspection of $\lk_\Delta v_i\cup\lk_\Delta v_j\cup\lk_\Delta v_k$ also implies that the boundary complexes of $S_i\cup S_j\cup S_k$'s cannot be triangulated 2-spheres for distinct $1\leq i,j,k\leq 4$.		
		
		\begin{proposition}
			Not all rank-selected subcomplexes of balanced simplicial spheres have ear decompositions.
		\end{proposition}
		\begin{proof}
			Consider the complex $\Delta$ constructed above. We denote the union of interior faces of a complex $\tau$ by $\intr\tau$. Suppose $\Delta\backslash  V_4$ has an ear decomposition $(\Gamma_1,\Gamma_2,\cdots, \Gamma_k)$. Since $|V_4|=5$ and $\beta_{2}(\Delta\backslash V_4)=4$, $k$ must be 4. Notice first that $\cup_{i\leq 4} \lk_\Delta v_i$ divides $\mathbb{R}^3$ into five subspaces, namely, $S_1,\cdots, S_4$ and the complement of $S_5$, each having $\lk_\Delta v_i$ as the boundary complex for $1\leq i\leq 5$ respectively. The complex $\Gamma_1$ is the union of 2-balls $B_1, B_2$ with $\partial B_1=\partial B_2=\Gamma_1\cap \Gamma_2$. By the Jordan theorem, $B_1\cup \Gamma_2$ is a triangulated 2-sphere that separates $\mathbb{R}^3$ into two connected components. Hence the bounded component must be either $S_i\cup S_j$ or $S_i\cup S_j\cup S_k$ for some $1\leq i,j,k\leq 4$. (We may assume that it is not $S_i$, since otherwise we may consider the 2-sphere $\cup_{i\leq 3}\Gamma_i-\cup_{1\leq i\neq j\leq 3}\intr(\Gamma_i\cap\Gamma_j)$ instead of $\Gamma_1\cup\Gamma_2-\intr(\Gamma_1\cap\Gamma_2)$, where the subset enclosed by this sphere in $\mathbb{R}^3$ cannot be $S_i$ anymore.) This contradicts the fact that the boundaries of $S_i\cup S_j$ or $S_i\cup S_j\cup S_k$ are not 2-spheres.
		\end{proof}
		
		\begin{remark}
			 One can think of all the figures illustrated above as projections of a subcomplex of $\Delta -\st_\Delta v_5$ onto $\mathbb{R}^3$. However, we do not know whether the complex provided in this section can be realized as the boundary of a 4-polytope.
		\end{remark}
		\section*{Acknowledgements}
		The author was partially supported by a graduate fellowship from NSF grant DMS-1361423. I thank Moritz Firsching for pointing out the automorphism groups of the constructions in Section 3 and 4 and running some computational tests to decide whether the constructions are polytopal. Many thanks to Lorenzo Venturello and the anonymous referees for pointing out mistakes in an earlier version and contributing to a few remarks, improvement of the proofs in this paper.
		\bibliographystyle{amsplain}
		
\end{document}